\documentclass[conference,letterpaper]{IEEEtran}

\usepackage[utf8]{inputenc} 
\usepackage[T1]{fontenc}
\usepackage{url}
\usepackage{ifthen}
\usepackage{cite}
\usepackage[cmex10]{amsmath} 
                             
\usepackage{amsthm}
\usepackage{graphicx}

\usepackage{mathrsfs}
\usepackage{bbm}
\usepackage{mathtools}

\theoremstyle{plain}
\newtheorem{theorem}{Theorem}
\newtheorem{lemma}{Lemma}

\newcommand{\ind}[1]{\mathbbm{1}_{\{#1\}}}   
\newcommand{\ccF}{\mathscr{F}}
\newcommand{\cJ}{\mathcal{J}}
\newcommand{\ccJ}{\mathcal{P}}
\newcommand{\TS}{\mathcal{S}}

\newcommand{\esinf}{{\rm ess}\,\!\inf}
\newcommand{\esup}{{\rm ess}\,\!\sup}

\newcommand{\Exp}{\mathsf{E}}
\newcommand{\Pro}{\mathsf{P}}
\newcommand{\LR}{\mathsf{L}}
\newcommand{\f}{\mathsf{f}}
\newcommand{\g}{\mathsf{g}}

\newcommand{\ignore}[1]{{}}

\newcommand{\pushright}[1]{\ifmeasuring@#1\else\omit\hfill$\displaystyle#1$\fi\ignorespaces}
\newcommand{\pushleft}[1]{\ifmeasuring@#1\else\omit$\displaystyle#1$\hfill\fi\ignorespaces}

\begin{document}
\title{Detecting Changes in Hidden Markov Models} 

\author{%
  \IEEEauthorblockN{George V.~Moustakides}
  \IEEEauthorblockA{Rutgers University and University of Patras\\
                    New Brunswick and Rion\\
                    NJ 08991, USA and 26500 Patras, Greece\\
                    Email: gm463@rutgers.edu and moustaki@upatras.gr}
}

\maketitle

\begin{abstract}
We consider the problem of sequential detection of a change in the statistical behavior of a hidden Markov model. By adopting a worst-case analysis with respect to the time of change and by taking into account the data that can be accessed by the change-imposing mechanism we offer alternative formulations of the problem. For each formulation we derive the optimum Shewhart test that maximizes the worst-case detection probability while guaranteeing infrequent false alarms.
\end{abstract}

\section{Introduction}
We consider a hidden Markov model (HMM) where $\{\xi_t\}_{t\geq1}$ is the observation process that is acquired sequentially and $\{z_t\}_{t\geq0}$ is a Markov process that controls the statistical behavior of $\{\xi_t\}$ but its state is hidden. Let also $\tau\in\{0,1,\ldots\}$ denote a \textit{changetime} with the processes $\{(\xi_t,z_t)\}$ following a nominal probability measure $\Pro_\infty$ up to and including time $\tau$ while, \textit{after} $\tau$, the probability measure switches to an alternative regime $\Pro_0$. This change induces a new measure which is denoted by $\Pro_\tau$ with $\Exp_\tau[\cdot]$ being reserved for the corresponding expectation. 

To be more precise, for $t\leq\tau$, we make the simplifying assumptions that the observations $\{\xi_t\}$ are i.i.d.~with a common pdf $\f_{\infty}(\xi)$ and the Markov process having a transition pdf $\g_{\infty}(z_t|z_{t-1})$. 
After the change the observations are conditionally independent and controlled by the Markov process. In particular $\xi_t$ conditioned on $z_t$ has a pdf $\f_0(\xi_t|z_t)$ while the Markov process has transition pdf $\g_0(z_t|z_{t-1})$. It is possible to have $\g_0(z_t|z_{t-1})=\g_{\infty}(z_t|z_{t-1})$, namely, the Markov process not to undergo any change. For simplicity, under the nominal measure the observations are assumed not to be controlled by the Markov process.

We would like to detect the onset of the change in the statistical behavior using a sequential strategy. We are therefore interested in defining a stopping time $T$ adapted to the filtration $\{\ccF^{\xi}_t\}_{t\geq0}$ generated by the observations, that is, $\ccF^{\xi}_t=\sigma\{\xi_1,\ldots,\xi_t\}$, to perform the detection. In order to select optimally $T$ we need to propose a suitable performance measure and properly optimize it. To derive our criterion we are going to extend the idea introduced in  \cite{Moustakides1}. Even though we can access only the observation process $\{\xi_t\}$ to perform detection, there is also a \textit{change-imposing mechanism} that must decide about the time $\tau$ to impose the change. And this mechanism may have access to \textit{a completely different set of data} to make this decision.

\section{Performance Measure}
As we suggested above $T$ is an $\{\ccF^\xi_t\}$-adapted stopping time. In order to capture the fact that the change-imposing mechanism may have access to completely different information, we are going to assume that $\tau$ is also a stopping time (the time that the data stop following the nominal model) but adapted to a filtration $\{\ccF^w_t\}$ where $\ccF^w_t=\sigma\{w_1,\ldots,w_t\}$. In other words, the change-imposing mechanism sequentially consults the data sequence $\{w_t\}$ and at each time instant \textit{makes a decision} as to whether it should impose a change or not. Both, ourselves that select $T$ and the change-imposing mechanism that selects $\tau$ are bound by a \textit{causality} constraint forbidding the use of any data from the future. Clearly process $\{w_t\}$, which is available to the change-imposing mechanism, may or may not include $\{\xi_t\}$ and it can be dependent or completely independent from the observations.

The change-imposing mechanism decides what is the best instant $\tau$ to impose the change while we decide what is the best time $T$ to stop and declare that a change took place. If $\phi(t,s)\geq0$ is a deterministic function expressing distance between or reward for the pair $(t,s)$ then we can use the conditional expectation $\Exp_\tau[\phi(T,\tau)|T>\tau]$ as a generic performance measure for the detection process. We condition on the event of no false alarms $\{T>\tau\}$ in order to compute the performance of $T$ only during successes since we intend to take care of false alarms differently.

Most of the time the rule that defines the stopping time $\tau$ is unknown, therefore the proposed performance measure cannot be computed and, more importantly, used to derive an optimum detection stategy. In such cases it is common to follow \textit{a worst-case analysis} with respect to $\tau$. In other words try to find the worst-case $\tau$ that will make the conditional expectation as unfavorable as possible to the detection goal. We have the following lemma that addresses this problem.
\begin{lemma}\label{lem:1}
Suppose that $T$ and $\tau$ are stopping times described as above, then
\begin{multline}
\inf_{\tau}\Exp_\tau[\phi(T,\tau)|T>\tau]\\
=\inf_{t\geq0}\esinf\Exp_t[\phi(T,t)|T>t,\ccF^w_t],
\label{eq:th1-1}
\end{multline}
The previous equality is also valid if we replace $\inf$ and $\esinf$ with $\sup$ and $\esup$.
\end{lemma}
\begin{proof}
The proof is given in the Appendix.
\end{proof}

If we select $\phi(t,s)=(t-s)^+$ then we can define an extension of Lorden's measure \cite{Lorden} by computing the worst-case average detection delay as follows
$$
\cJ(T)=\sup_{t\geq0}\esup\Exp_t[T-t|T>t,\ccF^w_t].
$$
We must emphasize that this is \textit{not} the original Lorden measure since conditioning is with respect to the sigma-algebra $\ccF^w_t$ that controls $\tau$ and not $\ccF^\xi_t$ used in the original definition. Furthermore, in our approach the double maximization occurs naturally as a result of our worst-case analysis and not because of some arbitrary definition.

An alternative measure can be generated by evaluating the performance of $T$ using the probability of detecting the change \textit{immediately} after it occurs. In other words we are interested in the probability of the event $\{T=\tau+1\}$. For this reason we define $\phi(t,s)=\ind{t=s+1}$. If we apply Lemma\,\ref{lem:1} we can compute the worst-case detection probability
\begin{equation}
\ccJ(T)=\inf_{t\geq0}\esinf\Pro_t(T=t+1|T>t,\ccF^w_t),
\label{eq:measure}
\end{equation}
which in this work is the criterion we intend to adopt.

Returning to HMMs and using \eqref{eq:measure} we distinguish \textit{four} different cases depending on how $\{w_t\}$ is related to the existing data.
\vskip0.1cm
\noindent i)~The change-imposing mechanism accesses information that is independent from $\{(\xi_t,z_t)\}$. In this case in \eqref{eq:measure} there is no conditioning with respect to $\ccF^w_t$ since the probability does not depend on $\{w_t\}$. This yields
\begin{equation}
\ccJ_{\text{i}}(T)=\inf_{t\ge0}\Pro_t(T=t+1|T>t).
\label{eq:criter1}
\end{equation}
This is the Pollak-like criterion proposed in \cite{Pollak2}. 
\vskip0.1cm
\noindent ii)~The change-imposing mechanism accesses only the observations, then
\begin{equation}
\ccJ_{\text{ii}}(T)=\inf_{t\ge0}\esinf\Pro_t(T=t+1|T>t,\ccF^{\xi}_t).
\label{eq:criter2}
\end{equation}
This is the Lorden-like criterion proposed in \cite{Moustakides2}.
\vskip0.1cm
\noindent iii)~The change-imposing mechanism accesses only the state of the Markov process. This leads to
\begin{equation}
\ccJ_{\text{iii}}(T)=\inf_{t\ge0}\esinf\Pro_t(T=t+1|T>t,\ccF^{z}_t),
\label{eq:criter3}
\end{equation}
where $\ccF^z_t=\sigma\{z_1,\ldots,z_t\}$ corresponding to $w_t=z_t$.
\vskip0.1cm
\noindent iv)~The change-imposing mechanism accesses both, the observations and the state of the Markov process resulting in
\begin{equation}
\ccJ_{\text{iv}}(T)=\inf_{t\ge0}\esinf\Pro_t(T=t+1|T>t,\ccF^{\xi,z}_t),
\label{eq:criter4}
\end{equation}
where $\ccF^{\xi,z}_t=\sigma\{\xi_1,\ldots,\xi_t,z_1,\ldots,z_t\}$ corresponding to $w_t=(\xi_t,z_t)$.

For each criterion we can define a constrained optimization problem whose solution will provide the optimum $T$:
\begin{equation}
\sup_{T}\ccJ_l(T),
~\text{subject to:}~\Exp_\infty[T]\geq\gamma>1,
\label{eq:max-min}
\end{equation}
where $l=\text{i,\,ii,\,iii,\,iv}$. In other words we maximize the worst-case detection probability assuring at the same time that the average period between false alarms is lower bounded by a constant $\gamma>1$ that we can select.

The idea of maximizing the detection probability was first introduced in \cite{Bojdecki} under Shiryaev's \cite{Shiryaev} Bayesian formulation. In \cite{Pollak2} we have a variant of the original Pollak measure \cite{Pollak1} while a variant of Lorden's measure \cite{Lorden} was adopted in \cite{Moustakides2} for independent processes and in \cite{Moustakides3} for Markov. In this work we address the case of HMMs. The problem of change-detection in HMMs has been considered in the past in \cite{Fhu1,Fhu2,Fhu3} and from these results it is well understood that even the asymptotic analysis is extremely challenging, not always leading to outcomes that are practically implementable.

\section{Candidate Tests}
Let us first present the joint data pdf induced by a change occurring at some time $\tau=t\geq0$. For $0<s\leq t$ we have
\begin{multline*}
\f_t(\xi_s,\ldots,\xi_1,z_s,\ldots,z_0)=\\[2pt]
\f_\infty(\xi_s)\cdots\f_\infty(\xi_1)\times\g_{\infty}(z_s|z_{s-1})\cdots\g_{\infty}(z_1|z_0)\g_{\infty}(z_0),
\end{multline*}
while for $s>t$ the resulting pdf takes the form
\begin{multline*}
\f_t(\xi_s,\ldots,\xi_1,z_s,\ldots,z_0)=\\[2pt]
\f_0(\xi_s|z_s)\cdots\f_0(\xi_{t+1}|z_{t+1})
\times\g_0(z_s|z_{s-1})\cdots\g_0(z_{t+1}|z_t)\\[2pt]
\times\f_\infty(\xi_t)\cdots\f_\infty(\xi_1)
\times\g_{\infty}(z_t|z_{t-1})\cdots\g_{\infty}(z_1|z_0)\g_{\infty}(z_0),
\end{multline*}
where $\g_{\infty}(z_0)$ is the marginal pdf of $z_0$. The pdfs $\f_{\infty}(\xi_t)$, $\f_0(\xi_t|z_t)$, $\g_{\infty}(z_{t}|z_{t-1})$, $\g_{\infty}(z_0)$ are assumed known.

To simplify our presentation we are going to assume that $\g_\infty(z_0)$ is the stationary pdf for the transition pdf $\g_\infty(z_t|z_{t-1})$, namely $\int \g_\infty(z_t|z_{t-1})\g_\infty(z_{t-1})dz_{t-1}=\g_\infty(z_t)$. We can then define the following average probability density
\begin{equation}\textstyle
\bar{\f}_0^1(\xi_t)=\iint\f_0(\xi_t|z_t)\g_0(z_t|z_{t-1})\g_\infty(z_{t-1})dz_{t-1}dz_t,
\label{eq:mbifla12}
\end{equation}
which, when $\g_0(z_t|z_{t-1})=\g_\infty(z_t|z_{t-1})$, simplifies to
\begin{equation}\textstyle
\bar{\f}_0^1(\xi_t)=\int\f_0(\xi_t|z_t)\g_\infty(z_t)dz_t,
\label{eq:mbifla122}
\end{equation}
and will be used for Criteria i) and ii). For Criteria iii) and iv) we define
\begin{equation}\textstyle
\bar{\f}_0^2(\xi_t)=\iint\f_0(\xi_t|z_t)\g_0(z_t|z_{t-1})\pi(z_{t-1})dz_{t-1}dz_t,
\label{eq:mbifla34}
\end{equation}
where $\pi(z)$ is a pdf to be specified in the sequel.

With the help of the average pdfs $\bar{\f}_0^j(\xi_t),~j=1,2$ we can now define the candidate Shewhart stopping time as follows
\begin{equation}
\LR_j(\xi_t)=\frac{\bar{\f}_0^j(\xi_t)}{\f_\infty(\xi_t)},~~\TS_j=\inf\{t>0:~\LR_j(\xi_t)\geq\nu_j\}.
\label{eq:LR12}
\end{equation}
Threshold $\nu_j>0$ is selected to satisfy the false alarm constraint with equality, namely
\begin{equation}\textstyle
\Exp_\infty[\TS_j]=\frac{1}{\Pro_\infty(\LR_j(\xi_t)\geq\nu_j)}=\gamma,
\label{eq:mbifla1}
\end{equation}
Existence of $\nu_j$ is guaranteed since the equation $\Pro_\infty(\LR_j(\xi_1)\geq\nu_j)=\frac{1}{\gamma}$ has always a solution if we assume that $\LR_j(\xi_t)$ does not contain any atoms under $\Pro_\infty$. Otherwise, in order to satisfy \eqref{eq:mbifla1} we may need randomization every time $\LR_j(\xi_t)=\nu_j$.

We can also compute the corresponding worst-case detection probability of the two Shewhart schemes. For the first test, since there is no dependence on the past, we have
\begin{equation}\textstyle
\beta_1=\ccJ_{\rm i}(\TS_1)=\ccJ_{\rm ii}(\TS_1)=\int\bar{\f}_0^1(\xi_t)\ind{\LR_1(\xi_t)\geq\nu_1}d\xi_t.
\label{eq:beta1}
\end{equation}
For the second Shewhart test the analysis for finding the worst-case detection probability is slightly more involved. Consider first the conditional pdf 
$$\textstyle
\f_0(\xi_t|z_{t-1})=\int\f_0(\xi_t|z_t)\g_0(z_t|z_{t-1})dz_t
$$
then the worst-case detection probability satisfies
\begin{multline}
\beta_2=\ccJ_{\rm iii}(\TS_2)=\ccJ_{\rm iv}(\TS_2)=\\
\inf_{z_{t-1}}{\textstyle\int}\f_0(\xi_t|z_{t-1})\ind{\LR_2(\xi_t)\geq\nu_2}d\xi_t.
\label{eq:beta2}
\end{multline}

We recall that the second Shewhart test is defined in terms of an arbitrary probability density $\pi(z)$. This means that the stopping time $\TS_2$ and also the worst-case detection probability $\beta_2$ are functions of $\pi(z)$ as well. To specify $\pi(z)$, let $\mathcal{Z}$ denote its support, then $\pi(z)$ must be such that
\begin{align}
\begin{split}
&{\textstyle\int}\f_0(\xi_t|z_{t-1})\ind{\LR_2(\xi_t)\geq\nu_2}d\xi_t=\beta_2,~\text{for}~z_{t-1}\in\mathcal{Z}\\
&{\textstyle\int}\f_0(\xi_t|z_{t-1})\ind{\LR_2(\xi_t)\geq\nu_2}d\xi_t\geq\beta_2,~\text{for}~z_{t-1}\not\in\mathcal{Z}.
\end{split}
\label{eq:mbifla2}
\end{align}
In other words, $\pi(z)$ must put all its probability mass onto points for which the Shewhart test exhibits its worst-case performance. In fact \eqref{eq:mbifla2} is sufficient to define $\pi(z)$ uniquely.

\section{Max-Min Optimality}
In this section we will demonstrate that the stopping times $\TS_j,~j=1,2$, defined in \eqref{eq:LR12} solve the max-min constrained optimization problem defined in \eqref{eq:max-min}. In order to prove our claim we first need to find a suitable upper bound for $\ccJ_l(T)$. The following theorem provides the necessary expressions.

\begin{theorem}\label{th:1}
For any stopping time $T$ with $\Exp_\infty[T]<\infty$ we have
\[
\ccJ_l(T)\textstyle\!\leq\!\frac{\Exp_\infty[\LR_1(\xi_T)]}{\Exp_\infty[T]},\,l={\rm i,ii};~~
\ccJ_l(T)\textstyle\!\leq\!\frac{\Exp_\infty[\LR_2(\xi_T)]}{\Exp_\infty[T]},\,l={\rm iii,iv}.
\]
Additionally, if $T=\TS_j,~j=1,2$, then we have equality in the corresponding inequality.
\end{theorem}
\vskip-0.1cm
\begin{proof} The proof is given in the Appendix.
\end{proof}

The next theorem optimizes the upper bounds proposed in Theorem\,\ref{th:1}.
\begin{theorem}\label{th:2}
If $T$ is any stopping time satisfying the false alarm constraint, then
\[
\textstyle\frac{\Exp_\infty[\LR_j(\xi_T)]}{\Exp_\infty[T]}\leq\beta_j,~j=1,2,
\]
where $\beta_1,\beta_2$ are defined in \eqref{eq:beta1}, \eqref{eq:beta2} respectively.
\end{theorem}
\vskip-0.1cm
\begin{proof} The proof is highlighted in the Appendix.
\end{proof}

Combining Theorems\,\ref{th:1} and \ref{th:2}, immediately assures optimality of the Shewhart tests. In particular for $l={\rm i,\,ii}$ we have
$$
\ccJ_l(T){\textstyle\leq\frac{\Exp_\infty[\LR_1(\xi_T)]}{\Exp_\infty[T]}}\leq\beta_1
=\ccJ_{\rm i}(\TS_1)=\ccJ_{\rm ii}(\TS_1),
$$
while for $l={\rm iii,\,iv}$ we conclude
$$
\ccJ_l(T){\textstyle\leq\frac{\Exp_\infty[\LR_2(\xi_T)]}{\Exp_\infty[T]}}\leq\beta_2
=\ccJ_{\rm iii}(\TS_2)=\ccJ_{\rm iv}(\TS_2).
$$
These two relationships establish optimality of the two Shewhart tests. In the next section we offer an example involving an interesting HMM.

\section{Example}
We consider the case of a Gaussian process whose mean is controlled by a Gaussian Markov process. Specifically,
let the observations $\{\xi_t\}$ before the change be i.i.d. with pdf $\f_\infty(\xi_t)\sim\mathcal{N}(0,1)$ and after the change assume $\f_0(\xi_t|z_t)\sim\mathcal{N}(z_t,1)$. The process $\{z_t\}$ is unobservable and of the form $z_t=\mu+v_t$ where $\mu$ is a constant denoting the mean of $z_t$ and $\{v_t\}$ is an AR(1) Gaussian process with $v_t$ being conditionally Gaussian of the form $v_t\sim\mathcal{N}(\alpha v_{t-1},\sigma^2)$ with $|\alpha|<1$.

For the stationary pdf we have $\g_\infty(z)\sim\mathcal{N}(\mu,\frac{\sigma^2}{1-\alpha^2})$. Since in this example we assume that the Markov process $\{z_t\}$ does not change, if we focus on the solution for Criteria i) and ii), we use \eqref{eq:mbifla122} to compute
\begin{equation}\textstyle
\bar{\f}_0^1(\xi)=\int\f_0(\xi|z)\g_\infty(z)dz~\sim\mathcal{N}\big(\mu,1+\frac{\sigma^2}{1-\alpha^2}\big).
\label{eq:ex0}
\end{equation}
Following \eqref{eq:LR12} we can easily establish that the optimal Shewhart test is equivalent to
\begin{equation}\textstyle
\TS_1=\inf\left\{t>0: \left|\xi_t+\mu\frac{1-\alpha^2}{\sigma^2}\right|\geq\nu_1\right\}.
\label{eq:Shewhart-i-ii}
\end{equation}
Threshold $\nu_1$ is related to the average false alarm period $\gamma$ through \eqref{eq:mbifla1} which takes the form
\begin{equation}\textstyle
\Phi\big(\mu\frac{1-\alpha^2}{\sigma^2}-\nu_1\big)+\Phi\big(-\mu\frac{1-\alpha^2}{\sigma^2}-\nu_1\big)=\frac{1}{\gamma},
\label{eq:ARL-i-ii}
\end{equation}
while the worst-case detection probability becomes
\begin{equation}\textstyle
\beta_1=\Phi\left(\frac{\mu\left(1+\frac{1-\alpha^2}{\sigma^2}\right)-\nu_1}{\sqrt{1+\frac{\sigma^2}{1-\alpha^2}}}\right)+\Phi\left(-\frac{\mu\left(1+\frac{1-\alpha^2}{\sigma^2}\right)+\nu_1}{\sqrt{1+\frac{\sigma^2}{1-\alpha^2}}}\right).
\label{eq:PD-i-ii}
\end{equation}

Let us now consider Criteria iii) and iv). We focus on the computation of \eqref{eq:mbifla34} and perform it in two steps. The first involves the computation of the conditional pdf
\begin{multline}\textstyle
\f_0(\xi_t|z_{t-1})=\int\f_0(\xi_t|z_t)\g_0(z_t|z_{t-1})dz_t\\
\sim\mathcal{N}\big((1-\alpha)\mu+\alpha z_{t-1},1+\sigma^2\big).
\label{eq:ex2}
\end{multline}
The next step consists in finding the pdf $\pi(z)$. We are going to assume that $\pi(z)$ puts all its mass on the single point $z=-\mu\frac{1-\alpha}{\alpha}$. This implies that $\bar{\f}_0^2(\xi)\sim\mathcal{N}(0,1+\sigma^2)$. We can then verify that the resulting Shewhart test is equivalent to
\begin{equation}
\TS_2=\inf\{t>0: |\xi_t|\geq\nu_2\}
\label{eq:Shewhart-iii-iv}
\end{equation}
with the threshold satisfying the false alarm constraint
\begin{equation}
2\Phi(-\nu_2)=\frac{1}{\gamma}.
\label{eq:ARL-iii-iv}
\end{equation}
Of course, in order for our selection of $\pi(z)$ to be correct we need to show validity of \eqref{eq:mbifla2}. Therefore we must prove that $\Pro_0(|\xi_t|\geq\nu_2|z_{t-1})$ has a minimum for $z_{t-1}=-\mu\frac{1-\alpha}{\alpha}$. Using \eqref{eq:ex2} the desired probability is
\begin{multline*}\textstyle
\Pro_0(|\xi_t|\geq\nu_2|z_{t-1})=
\Phi\left(\frac{-\nu_2+((1-\alpha)\mu+\alpha z_{t-1})}{\sqrt{1+\sigma^2}}\right)\\\textstyle
+\Phi\left(\frac{-\nu_2-((1-\alpha)\mu+\alpha z_{t-1})}{\sqrt{1+\sigma^2}}\right)
\end{multline*}
which is clearly minimized when $(1-\alpha)\mu+\alpha z_{t-1}=0$ with the minimum being equal to
\begin{equation}\textstyle
\beta_2=2\Phi\big(-\frac{\nu_2}{\sqrt{1+\sigma^2}}\big).
\label{eq:PD-iii-iv}
\end{equation}
The latter also constitutes the optimum worst-case detection probability for the Shewhart test in \eqref{eq:Shewhart-iii-iv}. It is worth mentioning that the Shewhart stopping time $\TS_2$ is UMP with respect to $\alpha,\mu$ and $\sigma^2$ since, as we can see, it does not require knowledge of these parameters. What is equally interesting is that the optimum worst-case detection probability $\beta_2$ is only a function of $\sigma^2$ and not of $\alpha,\mu$. It is only $\pi(z)$ that depends on these two parameters.

Suppose now that we erroneously assume that the change-imposing mechanism does not access the state of the Markov process when in reality it does. In this case we will be using $\TS_1$ from \eqref{eq:Shewhart-i-ii} instead of $\TS_2$ from \eqref{eq:Shewhart-iii-iv}. For $\TS_1$ it is not difficult to verify that the worst-case detection probability is equal to
\begin{equation}\textstyle
\tilde{\beta}_1=2\Phi\big(-\frac{\nu_1}{\sqrt{1+\sigma^2}}\big).
\label{eq:PD-i-ii-error}
\end{equation}

A similar erroneous assumption can occur when we consider the change-imposing mechanism to be able to access the Markov state when in reality it does not. Consequently by using $\TS_2$ from \eqref{eq:Shewhart-iii-iv} we need to compute its performance under the pdf in \eqref{eq:ex0}. This yields
\begin{equation}\textstyle
\tilde{\beta}_2=\Phi\left(\frac{(-\nu_2+\mu)\sqrt{1-\alpha^2}}{\sqrt{1-\alpha^2+\sigma^2}}\right)+
\Phi\left(-\frac{(\nu_2+\mu)\sqrt{1-\alpha^2}}{\sqrt{1-\alpha^2+\sigma^2}}\right).
\label{eq:PD-iii-iv-error}
\end{equation}
Clearly \eqref{eq:PD-iii-iv-error} must be compared against the optimum \eqref{eq:PD-i-ii} while \eqref{eq:PD-i-ii-error} against the optimum \eqref{eq:PD-iii-iv}. 

\begin{figure}[b]
\vskip-0.4cm
\centerline{\includegraphics[width=\hsize]{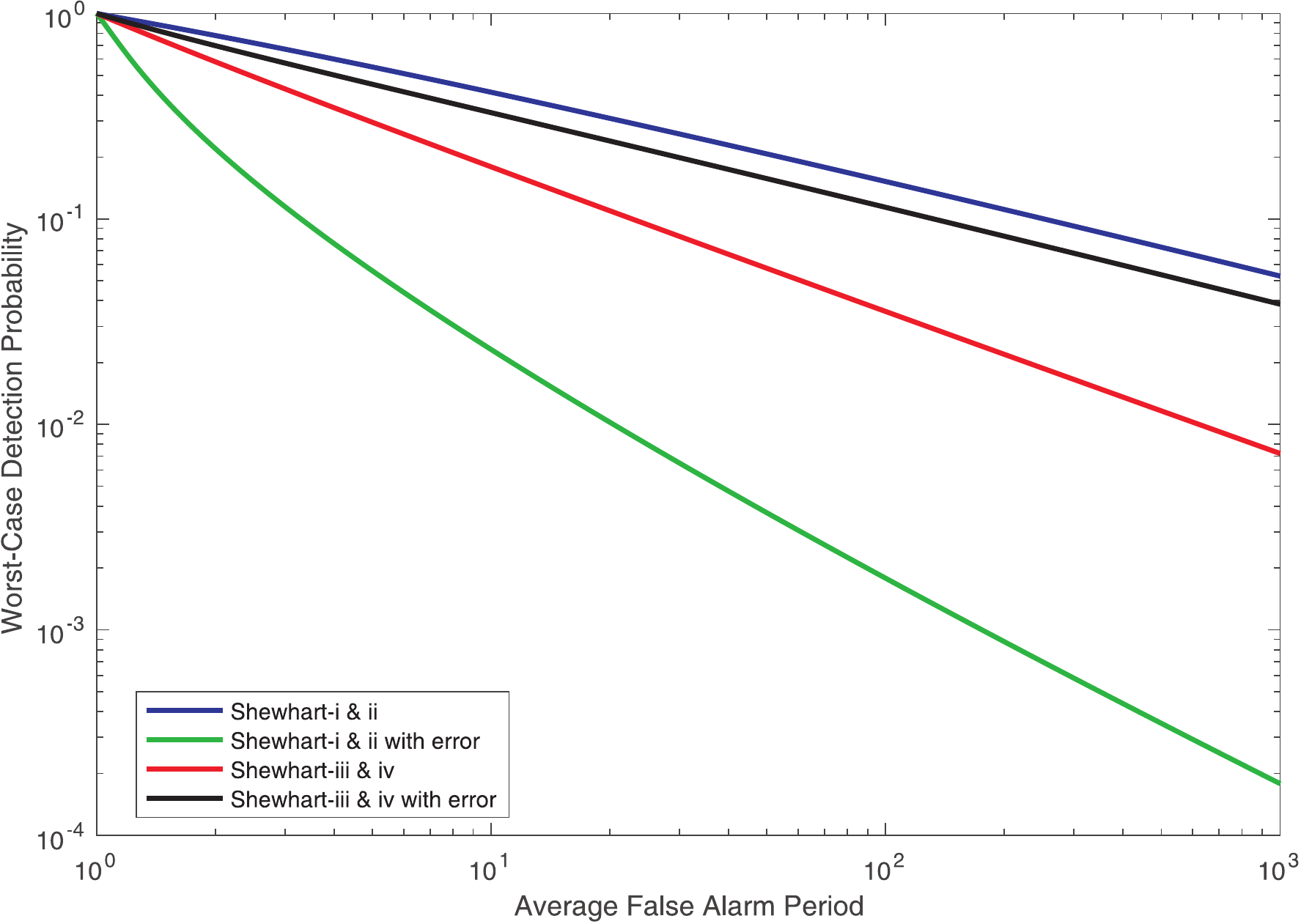}}
\caption{Detection probability as a function of average false alarm period of Shewhart test when change-imposing mechanism does not access the Markov process state and we correctly assume it does not (blue); when it does not and we erroneously assume it does (black); when it does and we correctly assume it does (red) and finally when it does and we erroneously assume it does not (green).}
\label{fig:1}
\end{figure}
For a numerical comparison, let $\alpha=0.5$, $\mu=1$, $\sigma^2=0.5$ with $\gamma$ ranging from 1 to 1000. Fig.\,\ref{fig:1} depicts the corresponding detection probabilities. The graph in blue corresponds to the change-imposing mechanism having no access to the Markov process and we correctly assume that it does not. This means that we plot $\beta_1$ from \eqref{eq:PD-i-ii} against $\gamma$ computed from \eqref{eq:ARL-i-ii}. If this assumption is wrong and the change-imposing mechanism can actually access the Markov state then we have a severe performance degradation depicted by the graph in green where we plot $\tilde{\beta}_1$ from \eqref{eq:PD-i-ii-error} against $\gamma$ from \eqref{eq:ARL-i-ii}. 

If we now use the test in \eqref{eq:Shewhart-iii-iv} and the change-imposing mechanism can indeed access the Markov state then the red graph depicts the worst-case detection probability $\beta_2$ from \eqref{eq:PD-iii-iv} as a function of $\gamma$ from \eqref{eq:ARL-iii-iv}. In case we made a mistake in our judgement and the change-imposing mechanism cannot access the Markov state then the same test has a performance depicted by the black curve where we plot $\tilde{\beta}_2$ from \eqref{eq:PD-iii-iv} in terms of $\gamma$ from \eqref{eq:ARL-iii-iv}. 

By using the Shewhart test $\TS_2$ in \eqref{eq:Shewhart-iii-iv}, which is obtained under more severe assumptions we do not lose much as compared to the optimum \eqref{eq:Shewhart-i-ii} if our assumption about the access capabilities of the change-imposing mechanism is incorrect. On the other hand, we guard ourselves against a hostile change-imposing mechanism when the latter can access all the available information. If, however, we assume that the change-imposing mechanism cannot access the Markov state and use $\TS_1$, this assumption can be catastrophic if it is wrong.

\section{Conclusion}
We considered the sequential change-detection problem for HMM which is known for being challenging. By introducing a generalized version of Lorden's performance measure we were able to come up with the optimum solution that maximizes the worst-case detection probability. This result is interesting since it is the first time we were able to obtain a solution for a performance measure that is different from the classical measures adopted so far in the literature.

\section*{Acknowledgement}
This work was supported by the US National Science Foundation under Grant CIF\,1513373, through Rutgers University.

\appendix
\noindent\textbf{Proof of Lemma\,\ref{lem:1}:} Since $\tau$ is a $\{\ccF^w_t\}$-adapted stopping time we have that $\{\tau=t\}$ is $\ccF^w_t$-measurable consequently we can write
\vskip-0.6cm
\begin{multline*}
\Exp_\tau[\phi(T,\tau)|T>\tau]=\\[3pt]\textstyle
\frac{\sum_{t=0}^\infty\Exp_\infty[\Exp_t[\phi(T,t)\ind{T>t}|\ccF^w_t]\ind{\tau=t}]}
{\sum_{t=0}^\infty\Exp_\infty[\Pro_t(T>t|\ccF^w_t)\ind{\tau=t}]}\geq\\[3pt]\textstyle
\inf_{t\geq0}\frac{\Exp_\infty[\Exp_t[\phi(T,t)\ind{T>t}|\ccF^w_t]\ind{\tau=t}]}
{\Exp_\infty[\Pro_t(T>t|\ccF^w_t)\ind{\tau=t}]}\geq\\[3pt]\textstyle
\inf_{t\geq0}\esinf\frac{\Exp_t[\phi(T,t)\ind{T>t}|\ccF^w_t]}
{\Pro_t(T>t|\ccF^w_t)}=\\[3pt]\textstyle
\inf_{t\geq0}\esinf\Exp_t[\phi(T,t)|T>t,\ccF^w_t].
\end{multline*}
This lower bound is in fact attainable. Suppose that the last double minimization is achieved by some $t_0$ (minimization over $t$) and realization $\{w_1,\ldots,w_{t_0}\}$ (minimization over the data), then the change-imposing mechanism can simply impose a change at $\tau=t_0$ when the specific combination of data occur. If there are more choices yielding the same lower bound then it can perform randomization between them.\hfill \QED 
\vskip0.1cm
\noindent\textbf{Proof of Theorem\,\ref{th:1}:} Let us consider first Criterion i). We have
\begin{multline*}\textstyle
\Pro_t(T=t+1|T>t)=\frac{\Pro_t(T=t+1)}{\Pro_t(T>t)}\\[4pt]\textstyle
=\frac{\Exp_\infty\big[\frac{\f_0(\xi_{t+1}|z_{t+1})\g_0(z_{t+1}|z_t)}{\f_\infty(\xi_{t+1})\g_\infty(z_{t+1}|z_t)}\ind{T=t+1}\big]}{\Pro_\infty(T>t)},
\end{multline*}
where the denominator takes this specific form because the event $\{T>t\}$ is $\ccF^\xi_t$-measurable and therefore happens before the change. Since $\{T=t+1\}$ is $\ccF^\xi_{t+1}$-measurable we need to average out $z_{t+1},\ldots,z_0$ conditioned on $\ccF^\xi_{t+1}$. This is easy since under $\Pro_\infty$ the observations and the Markov process are independent. Indeed this conditional expectation becomes
\begin{multline*}\textstyle
\Exp_\infty\left[\frac{\f_0(\xi_{t+1}|z_{t+1})\g_0(z_{t+1}|z_t)}{\f_\infty(\xi_{t+1})\g_\infty(z_{t+1}|z_t)}|\ccF_{t+1}^\xi\right]=\\ \textstyle
\int\!\frac{\f_0(\xi_{t+1}|z_{t+1})}{\f_\infty(\xi_{t+1})}\g_0(z_{t+1}|z_t)\g_\infty(z_t|z_{t-1})\cdots \g_\infty(z_0)dz_{t+1}\cdots dz_0\\
=\textstyle
\int\frac{\f_0(\xi_{t+1}|z_{t+1})}{\f_\infty(\xi_{t+1})}\g_0(z_{t+1}|z_t)\g_\infty(z_t)dz_{t+1}dz_t
=\frac{\bar{\f}^1_0(\xi_{t+1})}{\f_\infty(\xi_{t+1})},
\end{multline*}
where we used the fact that $\g_\infty(z)$ is the stationary pdf.
Since $\ccJ_{\rm i}(T)\leq\Pro_t(T=t+1|T>t)$ we can conclude that
$$
\textstyle
\ccJ_{\rm i}(T)\Pro_\infty(T>t)\leq\Exp_\infty\left[\frac{\bar{\f}_0^1(\xi_{t+1})}{\f_\infty(\xi_{t+1})}\ind{T=t+1}\right].
$$
Summing over $t=0,1,\ldots$ yields the desired inequality. The previous inequality becomes an equality when $T=\TS_1$ because the Shewhart test is an equalizer, namely, $\Pro_t(\TS_1=t+1|\TS_1>t)$ is a constant independent from $t$.

For Criterion ii) derivations are similar. Indeed we can write
$$
\textstyle
\ccJ_{\rm ii}(T)\Pro_\infty(T>t|\ccF_t^\xi)\leq\Exp_\infty\left[\frac{\bar{\f}^1_0(\xi_{t+1})}{\f_\infty(\xi_{t+1})}\ind{T=t+1}|\ccF_t^\xi\right].
$$
Taking expectation on both sides with respect to the $\Pro_\infty$ measure and summing over $t$ yields the desired result. Again we have equality when $T=\TS_1$ because $\TS_1$ is an equalizer.

Let us now consider Criterion iii), we have
$$\textstyle
\ccJ_{\rm iii}(T)\Exp_\infty[\ind{T>t}|\ccF^z_t]\leq\Exp_\infty\left[\frac{\f_0(\xi_{t+1}|z_{t+1})\g_0(z_{t+1}|z_t)}{\f_\infty(\xi_{t+1})\g_\infty(z_{t+1}|z_t)}|\ccF_t^z\right].
$$
Multiplying both sides with $\varpi(z_t)\geq0$ and averaging with respect to $\Pro_\infty$ yields
\begin{multline*}\textstyle
\ccJ_{\rm iii}(T)\Exp_\infty[\varpi(z_t)\ind{T>t}]\leq\\[1pt]\textstyle
\Exp_\infty\left[\frac{\f_0(\xi_{t+1}|z_{t+1})\g_0(z_{t+1}|z_t)}{\f_\infty(\xi_{t+1})\g_\infty(z_{t+1}|z_t)}\varpi(z_t)\ind{T=t+1}\right].
\end{multline*}
For the left hand side we have
\begin{multline*}
\Exp_\infty[\varpi(z_t)\ind{T>t}]=\Exp_\infty[\Exp_\infty[\varpi(z_t)|\ccF_t^\xi]\ind{T>t}]\\[1pt]
\textstyle
=\left(\int\varpi(z_t)\g_\infty(z_t)dz_t\right)\Exp_\infty[\ind{T>t}]=\Exp_\infty[\ind{T>t}],
\end{multline*}
where we define $\pi(z)=\varpi(z)\g_\infty(z)$ and, without loss of generality, we assume that $\int\pi(z_t)dz_t=1$. For the right hand side we can similarly write
\begin{multline*}\textstyle
\Exp_\infty\big[\frac{\f_0(\xi_{t+1}|z_{t+1})\g_0(z_{t+1}|z_t)}{\f_\infty(\xi_{t+1})\g_\infty(z_{t+1}|z_t)}\varpi(z_t)\ind{T=t+1}\big]=\\[1pt]\textstyle
\Exp_\infty\big[\Exp_\infty\big[\frac{\f_0(\xi_{t+1}|z_{t+1})\g_0(z_{t+1}|z_t)}{\f_\infty(\xi_{t+1})\g_\infty(z_{t+1}|z_t)}\varpi(z_t)|\ccF_{t+1}^\xi\big]\ind{T=t+1}\big]=\\[1pt]\textstyle
\Exp_\infty\big[\frac{\iint \f_0(\xi_{t+1}|z_{t+1})\g_0(z_{t+1}|z_t)\varpi(z_t)\g_\infty(z_t)dz_t dz_{t+1}}{\f_\infty(\xi_{t+1})}\ind{T=t+1}\big]\\[1pt]
\textstyle =\Exp_\infty\big[\frac{\bar{\f}^2_0(\xi_{t+1})}{\f_\infty(\xi_{t+1})}\ind{T=t+1}\big],
\end{multline*}
where in the last equality we use the definition in \eqref{eq:mbifla34}. The desired inequality can be shown as in the previous cases. Finally, when $T=\TS_2$ we have equality because $\pi(z)$ puts all its mass on values of $z_t$ where the essential infimum is attained and because the resulting value is independent from $t$ (equilizer). Similarly we can prove the upper bound for Criterion iv).\hfill\QED
\vskip0.1cm
\noindent\textbf{Proof of Theorem\,\ref{th:2}:} The first step in the proof consists in observing that we can limit ourselves to stopping times $T$ that satisfy the false alarm constraint with equality, that is, $\Exp_\infty[T]=\gamma$. Indeed if $\Exp[T]>\gamma$ then we can perform a randomization before taking any observations as to whether we should stop at time 0 with probability $p$ or continue according to the stopping time $T$ with probability $1-p$. This generates a new stopping $\tilde{T}$ that satisfies $\Exp_\infty[\tilde{T}]=(1-p)\Exp_\infty[T]$ and therefore we can select $p$ so that $\tilde{T}$ satisfies $\Exp_\infty[\tilde{T}]=\gamma$. On the other hand we can verify that
$$\textstyle
\frac{\Exp_\infty[\LR_j(\xi_{\tilde{T}})]}{\Exp_\infty[\tilde{T}]}=\frac{\Exp_\infty[\LR_j(\xi_{T})]}{\Exp_\infty[T]}.
$$

Because of the previous observations we need to prove that $\Exp_\infty[\beta_jT-\LR_j(\xi_{T})]\geq0$ over all $T$ satisfying the false alarm constraint with equality. In fact it will be sufficient if we consider the unconstrained version
\begin{equation}\textstyle
\Exp_\infty\big[(\beta_j-\frac{\nu_j}{\gamma})T-\LR_j(\xi_T)\big]\geq -\nu_j
\label{eq:app1}
\end{equation}
obtained by subtracting $\frac{\nu_j}{\gamma}\Exp_\infty[T]$ from the left and $\nu_j$ from the right side. We can now assume that there is no constraint on $T$ and minimize the left hand side in \eqref{eq:app1} over $T$. Since $T$ is $\{\ccF_t^\xi\}$-adapted and $\{\xi_t\}$ under $\Pro_\infty$ is i.i.d., this optimal stopping problem can be easily solved and we can show that the optimum stopping time is $\TS_j$ defined in \eqref{eq:LR12}. By direct computation we can also verify that the minimum value of the left hand side in \eqref{eq:app1} is indeed equal to $-\nu_j$.
\hfill\QED

\end{document}